\documentclass[12pt]{article}
\usepackage{graphicx}
\usepackage{amsmath}
\usepackage{amssymb}
\usepackage{amsthm}
\usepackage{xcolor}
\usepackage{tabularx}
\usepackage[toc,page]{appendix}
\usepackage{multirow}

\newtheorem{thm}{Theorem}

\newtheorem{lem}{Lemma}

\providecommand{\keywords}[1]
{
  \textit{Keywords--} #1
}

\title{Integers with four close factorizations}
\author{Tsz Ho Chan, Laura Holmes, Michael Liu, and Jose Villarreal}

\begin{document}

\begin{center}
\vskip 1cm{\LARGE\bf 
Numbers with Four Close Factorizations  
}
\vskip 1cm
\large
Tsz Ho Chan, Laura Holmes, Michael Liu, and Jose Villarreal \\
Mathematics Department \\
Kennesaw State University \\
Marietta, GA 30060 \\
USA \\
{\tt tchan4@kennesaw.edu, lholme10@students.kennesaw.edu, mliu2@students.kennesaw.edu, jvillar7@students.kennesaw.edu}
\end{center}

\vskip .2 in

\begin{abstract}
In this paper, we study numbers $n$ that can be factored in four different ways as $n = A B = (A + a_1) (B - b_1) = (A + a_2) (B - b_2) = (A + a_3) (B - b_3)$ with $B \le A$, $1 \le a_1 < a_2 < a_3 \le C$ and $1 \le b_1 < b_2 < b_3 \le C$. We obtain the optimal upper bound $A \le 0.04742 \ldots \cdot C^3 + O(C)$. The key idea is to transform the original question into generalized Pell equations $a x^2 - b y^2 = c$ and study their solutions.
\end{abstract}


\section{Introduction and main result}

Integer factorization is a fundamental study in number theory. Here we are interested in integers with close factorizations. For example,
\[
159600 = 399 \cdot 400 = 380 \cdot 420,
\]
and
\[
3950100 = 1900 \cdot 2079 = 1890 \cdot 2090 = 1881 \cdot 2100
\]
are numbers with two and three close factorizations respectively. One may ask if there is any structure for integers with close factorizations. Equivalently, one can interpret close factorizations of $n$ as close lattice points / integer points on the hyperbola $x y = n$. Previously, Cilleruelo and Jim\'{e}nez-Urroz \cite{CJ1, CJ2} and Granville and Jim\'{e}nez-Urroz \cite{GJ} studied close lattice points on hyperbolas and related questions. Suppose a positive integer $n$ can be factored in $k$ different ways:
\begin{equation} \label{factor1}
n = A B = (A + a_1) (B - b_1) = (A + a_2) (B - b_2) = \cdots = (A + a_{k-1}) (B - b_{k-1})
\end{equation}
for some integers
\begin{equation} \label{condition}
1 \le B \le A, \; \; 1 \le a_1 < a_2 < \cdots < a_{k-1} \le C, \; \text{ and } \; 1 \le b_1 < b_2 < \cdots < b_{k-1} \le C
\end{equation}
where $C$ is a certain parameter measuring ``closeness" of the factors.  A central question is to study the dependence of $A$ and $B$ in terms of $C$ and $k$. The first author \cite{C1} (and later in \cite{C2} which fixed an earlier mistake) proved the following upper bound for $A$ and $B$ when $k = 3$, and showed that no such upper bound exists when $k = 2$.
\begin{thm} \label{thm0}
Let $C \ge 5$. Suppose $n$ satisfies \eqref{factor1} and \eqref{condition} with $k = 3$. Then we have
\[
B \le A \le \frac{1}{4} C (C-1)^2.
\]
Moreover, the upper bound $\frac{1}{4} C (C-1)^2$ can be attained if and only if $C = 2N+1$ and
\begin{align*}
n =& [(2N+1) N^2 - (2N+1)] \cdot [(2N-1) (N-1)(N+1) + (2N-1)] \\
=& [(2N+1) N^2 - N] \cdot [(2N-1) (N-1)(N+1) + (N-1)] \\
=&  [(2N+1) N^2] \cdot [(2N-1)(N-1)(N+1)].
\end{align*}
\end{thm}

In this paper, we investigate the situation when $k = 4$. Based on Lemma \ref{lem1} from the next section, we know that $a_{k-1}$ is the largest among all the $a_i$'s and $b_i$'s. Hence, one may simply set $C = a_{k-1}$. We are interested in the ratio
\[
R_k := \frac{A}{a_{k-1}^3}.
\]
The above discussion tells us that $R_2$ does not exist while $R_3 \le 0.25$ where $0.25$ is the smallest possible upper bound. Our main result is the following.
\begin{thm} \label{thm1}
Suppose $n$ satisfies \eqref{factor1} and \eqref{condition} with $k = 4$. Then we have
\[
R_4 \le \frac{6 + \sqrt{6}}{9 (2 + \sqrt{6})^2} + O \Bigl( \frac{1}{n^{1/3}} \Bigr) \approx 0.04742\ldots
\]
when $n$ is sufficiently large. Moreover, the above bound is best possible.
\end{thm}
\noindent From the proof (using $x_2 = 49$ and $y_2 = 20$), we have the following nice numerical example:
\[
665165362680 = \underbrace{902460}_{A} \cdot \underbrace{737058}_{B} = \underbrace{902520}_{A + a_1} \cdot \underbrace{737009}_{B - b_1} = \underbrace{902629}_{A + a_2} \cdot \underbrace{736920}_{B - b_2} = \underbrace{902727}_{A + a_3} \cdot \underbrace{736840}_{B - b_3}
\]
with
\[
R_4 = \frac{902460}{(902727 - 902460)^3} = 0.0474126\ldots \, .
\]
Theorem \ref{thm1} shows intimate connection between numbers with four close factorizations and solutions of generalized Pell equations $a x^2 - b y^2 = c$. The interested readers can consult these  notes \cite{Co1, Co2} by Keith Conrad for more background on Pell-type equations for example.

\bigskip

This paper is organized as follows. First, we make some useful observations based on the four close factorizations. It includes the special situation $a_3 b_1 - a_1 b_3 = a_3 b_2 - a_2 b_3$ where we obtain the superior bound $B \le A \le 0.25 C^2$ (see Lemma \ref{lem8}). Then we work out the special case when $a_2 b_1 - a_1 b_2 = 1$. The crux of the method is to transform our original question into generalized Pell equations (see \eqref{pell-1} or \eqref{pell-1-g} for example). Based on this, we build up the general Pell-machinery. Then we study various scenarios (as limited by Lemmas \ref{lem4}, \ref{lem5}, \ref{lem6}, \ref{lem7}, and \ref{lem8}), and eliminate many of them through the usage to divisibility and modular arithmetic. Next, we derive a formula for the ratio $R_4$ when solutions to Pell-type equations exist. At the end, we combine everything to deduce Theorem \ref{thm1}.

\bigskip

{\bf Notation.} When an integer $a$ divides another integer $b$, we abbreviate it as $a \mid b$. Similarly, the symbol $a \nmid b$ means that $a$ does not divide $b$. The big-$O$ notation $f(x) = g(x) + O(h(x))$ means that $|f(x) - g(x)| \leq c h(x)$ for some constant $c > 0$. In particular, the expression $f(x) = O(g(x))$ means that $|f(x)| \le c g(x)$ for some constant $c > 0$.

\bigskip

{\bf Acknowledgment} This research is supported by the $2025$ Inspire Summer Scholars Program of the College of Science and Mathematics at Kennesaw State University.

 \section{Some initial observations}

First, we borrow some basic tools from \cite{C1}. Expanding $A B = (A + a_i) (B - b_i)$ and simplifying, we have $a_i B - b_i A = a_i b_i$. Dividing both sides by $b_i B$, we get
\begin{equation} \label{basic1}
\frac{a_i}{b_i} - \frac{A}{B} = \frac{a_i}{B}.
\end{equation}
Similarly, if one divides both sides by $a_i A$ instead, we get
\begin{equation} \label{basic1'}
\frac{B}{A} - \frac{b_i}{a_i} = \frac{b_i}{A}.
\end{equation}
Applying \eqref{basic1} with two different indices $i$, $j$ and subtracting the two equations, we have
\begin{equation} \label{basic2}
\frac{a_i}{b_i} - \frac{a_j}{b_j} = \frac{a_i - a_j}{B} \; \; \text{ or } \; \; a_i b_j - a_j b_i = \frac{b_i b_j (a_i - a_j)}{B}.
\end{equation}
Similarly, applying \eqref{basic1'} with two different indices $i$, $j$ and subtracting, we have
\begin{equation} \label{basic2'}
\frac{b_j}{a_j} - \frac{b_i}{a_i} = \frac{b_i - b_j}{A} \; \; \text{ or } \; \; a_i b_j - a_j b_i = \frac{a_i a_j (b_i - b_j)}{A}.
\end{equation}
Now, let us make a new definition which we call ``skews": 
\begin{equation} \label{disc}
D_{i j} := a_i b_j - a_j b_i  \; \; \text{ for } \; \; 1 \le j \neq i \le 3.
\end{equation}
Note that $D_{j i} = - D_{i j}$. Since $0 < a_1 < a_2 < a_3$ and $0 < b_1 <  b_2 < b_3$, we obtain
\[
D_{i j} = \frac{b_i b_j (a_i - a_j)}{B} = \frac{a_i a_j (b_i - b_j)}{A} > 0
\]
for all $1 \le j < i \le 3$ from \eqref{basic2} and \eqref{basic2'}. Then we can deduce that
\begin{equation} \label{basic3}
B = \frac{b_2 b_1 (a_2 - a_1)}{D_{21}} = \frac{b_3 b_1 (a_3 - a_1)}{D_{31}} = \frac{b_3 b_2 (a_3 - a_2)}{D_{32}},
\end{equation}
and
\begin{equation} \label{basic3'}
A = \frac{a_2 a_1 (b_2 - b_1)}{D_{21}} = \frac{a_3 a_1 (b_3 - b_1)}{D_{31}} = \frac{a_3 a_2 (b_3 - b_2)}{D_{32}}.
\end{equation}
\begin{lem} \label{lem1}
For $1 \le i \le k-1$, we have $a_i > b_i$. We also have $a_i - a_j > b_i - b_j$ for all $1 \le j < i \le k-1$.
\end{lem}
\begin{proof}
Equation \eqref{basic1} implies
\begin{equation} \label{temp1}
\frac{a_i}{b_i} = \frac{A + a_i}{B} > \frac{A}{B} \ge 1
\end{equation}
as $A \ge B$. This gives the first half of the lemma. Note: Using \eqref{basic1'} instead, one can obtain
\begin{equation} \label{temp2}
\frac{b_i}{a_i} = \frac{B - b_i}{A} < \frac{B}{A} \le 1.
\end{equation}
From \eqref{basic3} and \eqref{basic3'}, we have
\[
\frac{B}{A} = \frac{b_i b_j (a_i - a_j)}{a_i a_j (b_i - b_j)} \; \; \text{ or } \; \; \frac{a_i - a_j}{b_i - b_j} = \frac{B}{A} \cdot \frac{a_i}{b_i} \cdot \frac{a_j}{b_j}.
\]
Applying \eqref{temp1} to the above equation, we get
\[
\frac{a_i - a_j}{b_i - b_j} > \frac{B}{A} \cdot \frac{A}{B} \cdot \frac{A}{B} = \frac{A}{B} \ge 1
\]
which gives the second half of the lemma as $b_i > b_j$.
\end{proof}

Next, we notice two simple relations among the skews $D_{21}, D_{31}, D_{32}$.
\begin{lem} \label{lem3}
One has the identities: $D_{31} a_2 - D_{21} a_3 = D_{32} a_1$ and $\, D_{31} b_2 - D_{21} b_3 = D_{32} b_1$.
\end{lem}
\begin{proof}
From the definition of $D_{ij}$'s, we have
\[
D_{31} a_2 - D_{21} a_3 = (a_3 b_1 - a_1 b_3) a_2 - (a_2 b_1 - a_1 b_2) a_3 =  a_1 a_3 b_2 - a_1 a_2 b_3 = a_1 D_{32}.
\]
This gives the first identity. The second one follows in a similar manner.
\end{proof}

\begin{lem} \label{lem4}
We have the inequality $D_{31} > D_{21}$.
\end{lem}
\begin{proof}
Since $D_{32}, a_1 > 0$, Lemma \ref{lem3} yields $D_{31} a_2 - D_{21} a_3 > 0$ or $D_{31} a_2 > D_{21} a_3$. Hence, we have $D_{31} > D_{21} a_3 / a_2 > D_{21}$ as $a_3 > a_2$ by Lemma \ref{lem1}.
\end{proof}

\begin{lem} \label{lem5}
We have the inequality $D_{32} + D_{21} > D_{31}$.    
\end{lem}
\begin{proof}
From \eqref{basic3}, we have
\[
\frac{b_2 b_1 (a_2 - a_1)}{D_{21}} = \frac{b_3 b_1 (a_3 - a_1)}{D_{31}} \; \; \text{ or } \; \; D_{31} b_2 (a_2 - a_1) = D_{21} b_3 (a_3 - a_1).
\]
Subtracting $D_{21} b_2 (a_3 - a_1)$ from both sides, we get
\[
b_2 (D_{31} a_2 - D_{31} a_1 - D_{21} a_3 + D_{21} a_1) = D_{21} (b_3 - b_2) (a_3 - a_1) > 0.
\]
From Lemma \ref{lem3}, we have $-D_{21} a_3 = D_{32} a_1 - D_{31} a_2$. Substituting this into the above, we get
\[
b_2 (D_{31} a_2 - D_{31} a_1 + D_{32} a_1 - D_{31} a_2 + D_{21} a_1) = b_2 a_1 (D_{32} + D_{21} - D_{31}) > 0
\]
which gives the lemma as $b_2$ and $a_1$ are positive. 
\end{proof}

\begin{lem} \label{lem6}
Let $0 < a_1 < a_2 < a_3 \le C$ and $0 < b_1 < b_2 < b_3 \le C$. Suppose $a_i = (1 + \lambda) a_j$ for some $1 \le j < i \le 3$ and $\lambda > 0$. Then we have
\[
B \le A < \frac{\lambda C^3}{D_{ij} (1 + \lambda)^2}.
\]
In particular, if one of $D_{21}, D_{31}, D_{32}$ is greater than $5$, we have $B \le A \le \frac{C^3}{24} < 0.042 C^3$.
\end{lem}
\begin{proof}
Since $a_i = (1 + \lambda) a_j\leq C$, we have $a_j\leq \frac{C}{(1+\lambda)}$. Also, $a_i-a_j=(1+\lambda)a_j-a_j=\lambda a_j$. 
Thus, applying \eqref{basic3'} and Lemma \ref{lem1}, we obtain
    \begin{equation*}
            B \le A = \frac{a_ia_j(b_i-b_j)}{D_{ij}} < \frac{a_i a_j (a_i-a_j)}{D_{ij}} =\frac{a_ia_j^2\lambda}{D_{ij}} \leq \frac{\lambda C^3}{D_{ij}(1+\lambda)^2}
    \end{equation*}
which yields the first half of the lemma. By the arithmetic-mean and geometric-mean (AM-GM) inequality $(a + b) / 2 \ge \sqrt{a b}$, we have $(1 + \lambda)^2 \ge 4 \lambda$. Hence, if some $D_{ij} > 5$, the above inequality implies $B \le A < \frac{\lambda C^3}{6 \cdot (4 \lambda)} = \frac{C^3}{24}$, and we have the entire lemma.
\end{proof}

Another consequence of the above lemma is that we can bound $R_4$ for a special instance.

\begin{lem} \label{lem7}
If $D_{31} = 5$ and $D_{32} = 4$, we have $R_4 \le 0.04$.
\end{lem}
\begin{proof}
Lemma \ref{lem3} gives us the equation $D_{21} a_1 + 4 a_3 = 5 a_2$. This implies $a_3 < \frac{5}{4} a_2$ as $a_1, D_{21} > 0$. Hence, if $a_3 = (1 + \lambda) a_2$, then $0 < \lambda < \frac{1}{4}$. Applying Lemma \ref{lem6} and recalling $C = a_3$, we have
\[
A \le \frac{\lambda a_3^3}{D_{32} (1 + \lambda)^2}.
\]
By simple calculus, one can check that the function $\frac{\lambda}{(1 + \lambda)^2}$ is increasing on the interval $(0, 1)$. Thus, we obtain
\[
R_4 = \frac{A}{a_3^3} \le \frac{1/4}{4 (1 + 1/4)^2} = 0.04.
\]
\end{proof}

Finally, we explore what happens if we have equal skews.

\begin{lem} \label{lem8}
Suppose $D_{31}=D_{32}$ and recall \[ n \;=\;A B
\;=\;(A+a_1)(B-b_1)
\;=\;(A+a_2)(B-b_2)
\;=\;(A+a_3)(B-b_3)
\]
admit four close factorizations. Then, we have the relations
\begin{equation} \label{AtimesA+a3}
    A (A + a_3)\;=\;(A + a_1) (A + a_2) \; \; \text{ and } \; \;  B (B - b_3)\;=\;(B - b_1) (B - b_2).
\end{equation}
Moreover, we have the superior upper bound $B \le A \le C^2 / 4$.
\end{lem}

\begin{proof}
First, from $D_{31} = D_{32}$, we have
\begin{equation} \label{D-temp}
a_3 b_1 - a_1 b_3 = a_3 b_2 - a_2 b_3 \; \; \text{ or } \; \; a_3 (b_2 - b_1) = b_3 (a_2 - a_1)
\end{equation}
after some rearrangment. Next, the first equalities of equations \eqref{basic3} and \eqref{basic3'} tell us
\begin{equation*}
    \frac{b_2 b_1 (a_2 - a_1)}{B} = D_{21} = \frac{a_2 a_1 (b_2 - b_1)}{A} \; \; \text{ or } \; \; \frac{b_2 b_1 b_3 (a_2 - a_1)}{B} = \frac{a_2 a_1 b_3 (b_2 - b_1)}{A}.
\end{equation*}
Now, we apply \eqref{D-temp} to the above equation and get
\begin{equation} \label{ABab}
    \frac{b_2 b_1 a_3 (b_2 - b_1)}{B} = \frac{a_2 a_1 b_3 (b_2 - b_1)}{A} \; \; \text{ or } \; \; \frac{b_2}{a_2} \cdot \frac{b_1}{a_1} \cdot \frac{a_3}{b_3} = \frac{B}{A}.
\end{equation}
Furthermore, applying \eqref{temp1} to \eqref{ABab}, we have 
\begin{equation*}
    \frac{B}{A+a_2}\cdot \frac{B}{A+a_1}\cdot \frac{A+a_3}{B} = \frac{B}{A}
\end{equation*}
which gives the first half of \eqref{AtimesA+a3}. Similarly, applying \eqref{temp2} to \eqref{ABab}, we have the second half of \eqref{AtimesA+a3}. Finally, expanding the first equation of \eqref{AtimesA+a3} and simplifying, we obtain
\[
A (a_3 - a_2 - a_1) = a_1 a_2 > 0,
\]
and $a_3 \ge a_1 + a_2 + 1$. By the AM-GM inequality, we have
\[
C \ge a_3 > a_1 + a_2 \ge 2\sqrt{a_1a_2} = 2\sqrt{A(a_3-a_2-a_1)} \ge 2\sqrt{A}
\]
which implies the last part of the lemma.
\end{proof}

\section{The situation when $D_{21} = 1$}

First, we consider the special case $D_{21} = 1$. It forms the prototype of subsequent argument. One may restrict to $2 \le D_{31}, D_{32} \le 5$ by Lemmas \ref{lem4} and \ref{lem6}. Also, Lemmas \ref{lem5} and \ref{lem8} allow us to focus on $D_{32} > D_{31} - 1$ with $D_{31} \neq D_{32}$. From \eqref{disc}, we have
\begin{equation} \label{start1k}
\left\{ \begin{array}{l} a_2 b_1 - b_2 a_1 = 1, \\
a_3 b_1 - b_3 a_1 = D_{31}.
\end{array} \right.
\end{equation}
The first equation in \eqref{start1k} implies $\text{gcd}(b_1, b_2) = 1$. From \eqref{basic3}, we have
\begin{equation} \label{D32D31}
D_{32} b_1 (a_3 - a_1) = D_{31} b_2 (a_3 - a_2).
\end{equation}
Equation \eqref{D32D31} implies $b_1 \mid D_{31} b_2 (a_3 - a_2)$ and $b_2 \mid D_{32} b_1 (a_3 - a_1)$. Since $\text{gcd}(b_1, b_2) = 1$, we must have $b_1 \mid D_{31} (a_3 - a_2)$ and $b_2 \mid D_{32} (a_3 - a_1)$ by Euclid's lemma. Hence, we can write
\[
b_1 k_1 = D_{31} (a_3 - a_2) \; \; \text{ and } \; \; b_2 k_2 = D_{32} (a_3 - a_1).
\]
for some integers $k_1$ and $k_2$. Putting these into \eqref{D32D31}, we can deduce $k_1 = k_2 = k > 0$,
\begin{equation} \label{k-eq}
b_1 k = D_{31} (a_3 - a_2), \; \; \text{ and } \; \; b_2 k = D_{32} (a_3 - a_1).
\end{equation}
By a similar argument involving \eqref{basic3'} instead of \eqref{basic3}, we also have
\begin{equation} \label{m-eq}
a_1 m = D_{31} (b_3 - b_2) \; \; \text{ and } \; \; a_2 m = D_{32} (b_3 - b_1)
\end{equation}
for some integer $m > 0$. Next, by Lemma \ref{lem1}, we observe the following inequality between $k$ and $m$:
\begin{equation} \label{km-ineq-1}
k = \frac{D_{31} (a_3 - a_2)}{b_1} > \frac{D_{31} (a_3 - a_2)}{a_1} > \frac{D_{31}(b_3-b_2)}{a_1}=m.
\end{equation}
Multiplying the second equation in \eqref{k-eq} with the first equation in \eqref{m-eq}, we get
\[
b_2 a_1 k m = D_{32} D_{31} (a_3 - a_1) (b_3 - b_2).
\]
On the other hand, from the proof of Lemma \ref{lem5}, we have
\[
b_2 a_1 (D_{32} - D_{31} + 1) = (b_3 - b_2)(a_3 - a_1).
\]
Consequently, we have
\begin{equation} \label{km}
k m = D_{32} D_{31} (D_{32} - D_{31} + 1).
\end{equation}

Now, we apply Lemma \ref{lem3} and get
\begin{equation*}
a_3 = D_{31} a_2 - D_{32} a_1 \; \; \text{ and } \; \; b_3 = D_{31} b_2 - D_{32} b_1.
\end{equation*}
Putting these equations into the first equations in \eqref{k-eq} and \eqref{m-eq}, we obtain
\[
D_{31} \bigl( (D_{31} - 1) a_2 - D_{32} a_1 \bigr) = k b_1 \; \; \text{ and } \; \; D_{31} \bigl( (D_{31} - 1) b_2 - D_{32} b_1 \bigr)= m a_1.
\]
Solving for $a_2$ and $b_2$, we have
\[
a_2 = \frac{k b_1 + D_{31} D_{32} a_1}{D_{31} (D_{31} - 1)} \; \; \text{ and } \; \; b_2 = \frac{m a_1 + D_{31} D_{32} b_1}{D_{31} (D_{31} - 1)}.
\]
Finally, substituting the above expressions into $a_2 b_1 - a_1 b_2 = 1$, we transform our original question to the following generalized Pell equation:
\begin{equation} \label{pell-1}
k b_1^2 - m a_1^2 = D_{31} (D_{31} - 1).
\end{equation}

As a result, we have the following six cases as $2 \le D_{31} \le D_{32} \le 5$ and $D_{31} \neq D_{32}$, and we drop the case $D_{31} = 5$ and $D_{32} = 4$ in light of Lemma \ref{lem7}. In many circumstances, the equation \eqref{pell-1} has no integer solution by modular arithmetic.

\bigskip

Case 1: $D_{31} = 2$ and $D_{32} = 3$. Then equation \eqref{km} gives $k m = 12$.

\begin{center}
\begin{tabular}{|c|c|c|c|}
\hline $(k, m)$ as $k > m$ from \eqref{km-ineq-1} & Pell equation \eqref{pell-1} & Solution? \\
\hline $(12,1)$ & $12b^2-a^2=2$ & No, by $(\bmod \, 4)$\\
\hline $(6,2)$ & $6b^2-2a^2=2$ & No, by $(\bmod \, 3)$\\
\hline $(4,3)$ & $4b^2-3a^2=2$ & No, by $(\bmod \, 4)$\\
\hline
\end{tabular}
\newline \small{Table 1: List out all Pell-type equations \eqref{pell-1} with solutions when $D_{31} = 2$, $D_{32} = 3$.}
\end{center}

\noindent For example, the first equation above $12 b^2 - a^2 = 2$ becomes $0 - a^2 \equiv 2$ (mod $4$) which has no solution as $x^2 \equiv 0$ or $1$ (mod $4$). Similarly, the second equation above $6 b^2 - 2 a^2 = 2$ becomes $0 + a^2 \equiv 2$ (mod $3$) which no solution as $x^2 \equiv 0$ or $1$ (mod $3$).

\bigskip

Case 2: $D_{31} = 2$ and $D_{32} = 4$. Then equation \eqref{km} gives $k m = 24$.

\begin{center}
\begin{tabular}{|c|c|c|c|}
\hline $(k, m)$ as $k > m$ & Pell equation \eqref{pell-1} & Solution? \\
\hline $(24,1)$ & $24b_1^2-a_1^2=2$ & No, by $(\bmod \, 4)$\\
\hline $(12,2)$ & $12b_1^2-2a_1^2=2$ & No, by $(\bmod \, 3)$\\
\hline $(8,3)$ & $8b_1^2-3a_1^2=2$ & No, by $(\bmod \, 4)$\\
\hline $(6,4)$ & $6b_1^2-4a_1^2=2$ & $b_1=1$ and $a_1=1$\\
\hline
\end{tabular}
\newline \small{Table 2: List out all Pell-type equations \eqref{pell-1} with solutions when $D_{31} = 2$, $D_{32} = 4$.}
\end{center}

Case 3: $D_{31} = 2$ and $D_{32} = 5$. Then equation \eqref{km} gives $k m = 40$.

\begin{center}
\begin{tabular}{|c|c|c|c|}
\hline $(k, m)$ as $k > m$ & Pell equation \eqref{pell-1} & Solution? \\
\hline $(40,1)$ & $40b_1^2-a_1^2=2$ & No, by $(\bmod \, 4)$\\
\hline $(20,2)$ & $20b_1^2-2a_1^2=2$ & $b_1=1$ and $a_1=3$\\
\hline $(10,4)$ & $10b_1^2-4a_1^2=2$ & No, by $(\bmod \, 5)$\\
\hline $(8,5)$ & $8b_1^2-5a_1^2=2$ & No, by $(\bmod \, 4)$\\
\hline
\end{tabular}
\newline \small{Table 3: List out all Pell-type equations \eqref{pell-1} with solutions when $D_{31} = 2$, $D_{32} = 5$.}
\end{center}

\noindent The third equation above $10b_1^2-4a_1^2 = 2$ becomes $0 + a_1^2 \equiv 2$ (mod $5$) which has no solution as $x^2 \equiv 0, 1, 4$ (mod $5$). The conclusion for the first and last equations follows similar argument as in case 1. By similar reasoning, we have the following two tables.

\bigskip

Case 4: When $D_{31} = 3$ and $D_{32} = 4$. Then equation \eqref{km} gives $k m = 24$.

\begin{center}
\begin{tabular}{|c|c|c|c|}
\hline $(k, m)$ as $k > m$ & Pell equation \eqref{pell-1} & Solution? \\
\hline $(24,1)$ & $24b_1^2-a_1^2=6$ & No, by $(\bmod \, 4)$\\
\hline $(12,2)$ & $12b_1^2-2a_1^2=6$ & No, by $(\bmod \, 5)$\\
\hline $(8,3)$ & $8b_1^2-3a_1^2=6$ & No, by $(\bmod \, 4)$\\
\hline $(6,4)$ & $6b_1^2-4a_1^2=6$ & $b_1=5$ and $a_1=6$\\
\hline
\end{tabular}
\newline \small{Table 4: List out all Pell-type equations \eqref{pell-1} with solutions when $D_{31} = 3$, $D_{32} = 4$.}
\end{center}

Case 5: $D_{31} = 3$ and $D_{32} = 5$. Then equation \eqref{km} gives $k m = 45$.

\begin{center}
\begin{tabular}{|c|c|c|c|}
\hline $(k, m)$ as $k > m$ & Pell equation \eqref{pell-1} & Solution? \\
\hline $(45,1)$ & $45b_1^2-a_1^2=6$ & No, by $(\bmod \, 4)$\\
\hline $(15,3)$ & $15b_1^2-3a_1^2=6$ & No, by $(\bmod \, 4)$\\
\hline $(9,5)$ & $9b_1^2-5a_1^2=6$ & No, by $(\bmod \, 4)$\\
\hline
\end{tabular}
\newline \small{Table 5: List out all Pell-type equations \eqref{pell-1} with solutions when $D_{31} = 3$, $D_{32} = 5$.}
\end{center}

%

\section{General Pell-machinery}

In general, the skew $D_{21}$ may not be $1$. Then we cannot conclude $\gcd(b_1, b_2) = \gcd(a_1, a_2) = 1$ and apply Euclid's lemma. In order to extend our previous argument to more general $D_{21}$, we introduce the following notation. Let
\begin{equation} \label{definitions}
d_a := \gcd(a_1, a_2), \; d_b := \gcd(b_1, b_2), \; \left\{ \begin{array}{l}
a_1 := d_a \alpha_1 \\
a_2 := d_a \alpha_2,
\end{array} \right. \; 
\left\{ \begin{array}{l}
b_1 := d_b \beta_1 \\
b_2 := d_b \beta_2,
\end{array} \right.
\end{equation}
Then we have $\gcd(\alpha_1, \alpha_2) = 1 = \gcd(\beta_1, \beta_2)$. Now, we are ready to generalize the argument in the previous section. From \eqref{disc} and \eqref{definitions}, we have
\begin{equation} \label{start1k-g}
\left\{ \begin{array}{l} \alpha_2 \beta_1 - \beta_2 \alpha_1 = \frac{D_{21}}{d_a d_b}, \\
d_b a_3 \beta_1 - d_a b_3 \alpha_1 = D_{31}.
\end{array} \right.
\end{equation}
From \eqref{basic3}, we have
\begin{equation} \label{D32D31-g}
D_{32} \beta_1 (a_3 - d_a \alpha_1) = D_{31} \beta_2 (a_3 - d_a \alpha_2).
\end{equation}
Equation \eqref{D32D31-g} implies $\beta_1 \mid D_{31} \beta_2 (a_3 - d_a \alpha_2)$ and $\beta_2 \mid D_{32} \beta_1 (a_3 - d_a \alpha_1)$. Since $\text{gcd}(\beta_1, \beta_2) = 1$, we have $\beta_1 \mid D_{31} (a_3 - d_a \alpha_2)$ and $\beta_2 \mid D_{32} (a_3 - d_a \alpha_1)$ by Euclid's lemma. Say
\[
\beta_1 k_1 = D_{31} (a_3 - d_a \alpha_2) \; \; \text{ and } \; \; \beta_2 k_2 = D_{32} (a_3 - d_a \alpha_1).
\]
Putting these equations into \eqref{D32D31-g}, we have $k_1 = k_2 = k > 0$,
\begin{equation} \label{k-eq-g}
\beta_1 k = D_{31} (a_3 - d_a \alpha_2) \; \; \text{ and } \; \; \beta_2 k = D_{32} (a_3 - d_a \alpha_1).
\end{equation}
By a similar argument with \eqref{basic3'} instead of \eqref{basic3}, we obtain
\begin{equation} \label{m-eq-g}
\alpha_1 m = D_{31} (b_3 - d_b \beta_2) \; \; \text{ and } \; \; \alpha_2 m = D_{32} (b_3 - d_b \beta_1)
\end{equation}
for some integer $m > 0$. By Lemma \ref{lem1}, we arrive at the inequality (similar to \eqref{km-ineq-1})
\begin{equation} \label{km-ineq-g}
k = \frac{d_b D_{31} (a_3 - d_a \alpha_2)}{d_b \beta_1} > \frac{d_b D_{31} (a_3 - d_a \alpha_2)}{a_1} > \frac{d_b D_{31}(b_3 - d_b \beta_2)}{d_a \alpha_1} = \frac{d_b}{d_a} m.
\end{equation}
Multiplying the second equation in \eqref{k-eq-g} with the first equation in \eqref{m-eq-g}, we get
\[
\beta_2 \alpha_1 k m = D_{32} D_{31} (a_3 - d_a \alpha_1) (b_3 - d_b \beta_2).
\]
On the other hand, from the proof of Lemma \ref{lem5}, we have
\[
d_a d_b \beta_2 \alpha_1 (D_{32} - D_{31} + D_{21}) = D_{21} (b_3 - d_b \beta_2)(a_3 - d_a \alpha_1).
\]
Consequently, we obtain
\begin{equation} \label{km-g}
k m = \frac{d_a d_b}{D_{21}}{D_{32} D_{31} (D_{32} +D_{21} - D_{31})} .
\end{equation}

Now, we apply Lemma \ref{lem3} and get
\begin{equation} \label{a3b3}
a_3 = \frac{d_a (D_{31} \alpha_2 - D_{32} \alpha_1)}{D_{21}} \; \; \text{ and } \; \; b_3 = \frac{d_b (D_{31} \beta_2 - D_{32} \beta_1)}{D_{21}}.
\end{equation}
Putting \eqref{a3b3} into the first equations in \eqref{k-eq-g} and \eqref{m-eq-g}, we get
\[
d_a D_{31} [(D_{31} - D_{21}) \alpha_2 - D_{32} \alpha_1] = D_{21} k \beta_1
\]
and
\[
d_b D_{31} [(D_{31} - D_{21}) \beta_2 - D_{32} \beta_1] = D_{21} m \alpha_1.
\]
Solving for $\alpha_2$ and $\beta_2$, we have
\begin{equation} \label{alpha2beta2}
\alpha_2 = \frac{D_{21} k \beta_1 + d_a D_{31} D_{32} \alpha_1}{d_a D_{31} (D_{31} - D_{21})} \; \; \text{ and } \; \; \beta_2 = \frac{D_{21} m \alpha_1 + d_b D_{31} D_{32} \beta_1}{d_b D_{31} (D_{31} - D_{21})}.
\end{equation}
Finally, substituting \eqref{alpha2beta2} into the first equation in \eqref{start1k-g}, we obtain the Pell-type equation:
\begin{equation} \label{pell-1-g}
d_b k \beta_1^2 - d_a m \alpha_1^2 = D_{31} (D_{31} - D_{21}).
\end{equation}

\section{The situation when $D_{21} = 2$}

Besides using $x^2 \equiv 0, 1$ (mod $3$), $x^2 \equiv 0, 1$ (mod $4$), and $x^2 \equiv 0, 1, 4$ (mod $5$), we need two more lemmas to rule out integer solutions to some Pell-type equations.

\begin{lem} \label{lem-Pell-check-1}
Let $K$, $M$, and $\tau$ be integers. Suppose, for some prime $p$ and some odd number $b > 0$, the equation 
\begin{equation} \label{temp-pell1}
    K x^2 - M y^2=\tau
\end{equation}
satisfies the conditions: (i) $p^{b}|\tau$, (ii) $p^{b+1}|K$, (iii) $p\nmid M$, (iv) $p^{b+1}\nmid\tau$, then no integer solution $(x,y)$ to \eqref{temp-pell1} exists.
\end{lem}
\begin{proof}
Suppose for contrary that there exist integers $x,y$ satisfying $Kx^2-My^2=\tau$. Since $p^b \mid \tau$, and $p^{b+1} \mid K$, there exist integers $\tau'$ and $K'$ such that $\tau=\tau'p^b$ and $K=K'p^{b+1}$. So, equation \eqref{temp-pell1} becomes
\begin{equation} \label{temp-temp1}
    K' p^{b+1} x^2- M y^2=\tau' p^b.
\end{equation}
Since $p \nmid M$, we have $p^b \mid y^2$. Now, suppose $y = p^r y'$ for some $p \nmid y'$ and positive integer $r$. Then we have $p^b \mid p^{2r} y'^2$ which implies $2r \ge b$ or $r \ge b/2$. However, since $b$ is odd, we must have $r \ge (b + 1)/2$. Then $p^{b+1}$ divides the left-hand side of \eqref{temp-temp1} but not its right-hand side. This contradiction gives the lemma.
\end{proof}

\begin{lem} \label{lem-Pell-check-2}
Let $K$, $M$, and $\tau$ be integers. Suppose, for some prime $p$, the equation 
\begin{equation} \label{temp-pell2}
    Kx^2-My^2=\tau  
\end{equation}
satisfies the conditions: (i) $p \mid K$ and $p^2 \nmid K$ (i.e., $K = p K'$ with $p \nmid K'$), (ii) $p \mid \tau$ (i.e., $\tau = p \tau'$), (iii) $p\nmid M$, and (iv) $(K')^{-1}\tau'$ is a quadratic non-residue (mod $p$), then no integer solution $(x,y)$ to \eqref{temp-pell2} exists. Here $K'^{-1}$ stands for the multiplicative inverse of $K' \pmod{p}$.
\end{lem}

\begin{proof}
Since $p \mid K$, $p \mid \tau$, and $p \nmid M$, we must have $p \mid y$. Say $y = p y'$ for some integer $y'$. Then equation \eqref{temp-pell2} can be rewritten as
\begin{equation*}
    p K'x^2-M p^2 y'^2=p \tau' \; \; \text{ or } \; \; K' x^2- M p y'^2=\tau'.
\end{equation*}
Reducing everything (mod $p$), we get 
\begin{equation*}
    K' x^2\equiv\tau' \pmod{p} \; \; \text{ or } \; \; x^2\equiv(K')^{-1} \tau' \pmod{p}
\end{equation*}
which has no integer solution $x$ as $(K')^{-1}\tau'$ is a quadratic non-residue (mod $p$).
\end{proof}

Suppose $D_{21} = 2$. We may restrict our attention to $3 \le D_{31} \le 5$ by Lemma \ref{lem4} and \ref{lem6}. Also, Lemma \ref{lem5} tells us that $D_{32} \ge D_{31} - 1$. From \eqref{disc}, \eqref{definitions}, \eqref{km-g} and \eqref{pell-1-g}, we have
\begin{equation} \label{2start1k}
\left\{ \begin{array}{l} \alpha_2 \beta_1 - \beta_2 \alpha_1 = \frac{2}{d_a d_b}, \\
d_b a_3 \beta_1 - d_a b_3 \alpha_1 = D_{31},
\end{array} \right.
\end{equation}
\begin{equation} \label{km-2}
k m = \frac{d_a d_b}{2}{D_{32} D_{31} (D_{32} - D_{31} + 2)},
\end{equation}
and
\begin{equation} \label{pell-1-2}
d_b k \beta_1^2 - d_a m \alpha_1^2 = D_{31} (D_{31} - 2).
\end{equation}
\noindent We have the following six cases as $3 \le D_{31} \le 5$, $D_{32} \ge D_{31} - 1$, and $D_{31} \neq D_{32}$ by Lemmas \ref{lem7} and \ref{lem8}. Note that $d_a d_b \mid 2$ since the left-hand side of \eqref{2start1k} is an integer.

\bigskip

Case 1: When $D_{31} = 3$ and $D_{32} = 2$. Then equation \eqref{km-2} gives $k m = 3 d_a d_b$.

\begin{center}
\begin{tabular}{|c|c|c|c|}
\hline
$(d_a, d_b)$ & $(k, m)$ as $k > \frac{d_a}{d_b} m$ & Pell equation \eqref{pell-1-2} & Solution  \\
\hline
$(1, 1)$ & $(3, 1)$ & $3 \beta_1^2 - \alpha_1^2 = 3$ & $\beta_1 = 2, \alpha_1 = 3$ \\
\hline
$(2,1)$ & $(6,1)$ & $6 \beta_1^2 - 2 \alpha_1^2 = 3$ & No, by parity \\
\hline & $(3,2)$ & $3 \beta_1^2 - 4 \alpha_1^2 = 3$ & $\beta_1 = 7$, $\alpha_1 = 6$ \\
\hline & $(2,3)$ & $2 \beta_1^2 - 6 \alpha_1^2 = 3$ & No, by parity \\
\hline
$(1,2)$ & $(6,1)$ & $12 \beta_1^2 - \alpha_1^2 = 3$ & $\beta_1 = 1$, $\alpha_1 = 3$ \\
\hline
\end{tabular}
\newline \small{Table 6: List out all Pell-type equations \eqref{pell-1-2} with solutions when $D_{31} = 3$, $D_{32} = 2$.}
\end{center}

Case 2: $D_{31} = 3$ and $D_{32} = 4$. Then equation \eqref{km-2} gives $k m = 18 d_a d_b$. We apply Lemma \ref{lem-Pell-check-1} with $p = 3$ for some of the checks below.

\begin{center}
\begin{tabular}{|c|c|c|c|}
\hline
$(d_a, d_b)$ & $(k, m)$ as $k > \frac{d_a}{d_b} m$ & Pell equation \eqref{pell-1-2} & Solution? \\
\hline $(1,1)$ & $(18,1)$ & $18\beta_1^2-\alpha_1^2=3$ & No, by Lemma \ref{lem-Pell-check-1} \\
\hline & $(9,2)$ & $9\beta_1^2-2\alpha_1^2=3$ & No, by Lemma \ref{lem-Pell-check-1} \\
\hline & $(6,3)$ & $6\beta_1^2-3\alpha_1^2=3$ & $\beta_1=1, \alpha_1=1$\\
\hline $(2,1)$ & $(36,1)$ & $36\beta_1^2-2\alpha_1^2=3$ & No, by parity\\
\hline & $(18,2)$ & $18\beta_1^2-4\alpha_1^2=3$ & No, by parity\\
\hline & $(12,3)$ & $12\beta_1^2-6\alpha_1^2=3$ & No, by parity \\
\hline & $(9,4)$ & $9\beta_1^2-8\alpha_1^2=3$ & No, by Lemma \ref{lem-Pell-check-1}\\
\hline & $(6,6)$ & $6\beta_1^2-12\alpha_1^2=3$ & No, by parity \\
\hline $(1,2)$ & $(36,1)$ & $72\beta_1^2-\alpha_1^2=3$ & No, by Lemma \ref{lem-Pell-check-1}\\
\hline & $(18,2)$ & $36\beta_1^2-2\alpha_1^2=3$ & No, by parity\\
\hline & $(12,3)$ & $24\beta_1^2-3\alpha_1^2=3$ & No, by $(\bmod \, 4)$\\
\hline & $(9,4)$ & $18\beta_1^2-4\alpha_1^2=3$ & No, by parity \\
\hline
\end{tabular}
\newline \small{Table 7: List out all Pell-type equations \eqref{pell-1-2} with solutions when $D_{31} = 3$, $D_{32} = 4$.}
\end{center}

Case 3: When $D_{31} = 3$ and $D_{32} = 5$. Then equation \eqref{km-2} gives $k m = 30 d_a d_b$. We apply Lemmas \ref{lem-Pell-check-1} and \ref{lem-Pell-check-2} with $p = 3$ for some of the checks below.

\begin{center}
\begin{tabular}{|c|c|c|c|}
\hline $(d_a, d_b)$ & $(k, m)$ as $k > \frac{d_a}{d_b} m$ & Pell equation \eqref{pell-1-2} & Solution? \\
\hline $(1,1)$ & $(30,1)$ & $30\beta_1^2-\alpha_1^2=3$ & No, by$(\bmod \, 5)$\\
\hline & $(15,2)$ & $15\beta_1^2-2\alpha_1^2=3$ & No, by Lemma \ref{lem-Pell-check-2}\\
\hline & $(10,3)$ & $10\beta_1^2-3\alpha_1^2=3$ & No, by Lemma \ref{lem-Pell-check-1}\\
\hline & $(6,5)$ & $6\beta_1^2-5\alpha_1^2=3$ & No, by $(\bmod \, 5)$\\
\hline $(2,1)$ & $(60,1)$ & $60\beta_1^2-2\alpha_1^2=3$ & No, by parity\\
\hline & $(30,2)$ & $30\beta_1^2-4\alpha_1^2=3$ & No, by parity\\
\hline & $(20,3)$ & $20\beta_1^2-6\alpha_1^2=3$ & No, by parity\\
\hline & $(15,4)$ & $15\beta_1^2-8\alpha_1^2=3$ & No, by Lemma \ref{lem-Pell-check-2}\\
\hline & $(12,5)$ & $12\beta_1^2-10\alpha_1^2=3$ & No, by parity\\
\hline & $(10,6)$ & $10\beta_1^2-12\alpha_1^2=3$ & No, by parity\\
\hline & $(6,10)$ & $6\beta_1^2-20\alpha_1^2=3$ & No, by parity\\
\hline $(1,2)$ & $(60,1)$ & $120\beta_1^2-\alpha_1^2=3$ & No, by $(\bmod \, 5)$\\
\hline & $(30,2)$ & $60\beta_1^2-2\alpha_1^2=3$ & No, by parity\\
\hline & $(20,3)$ & $40\beta_1^2-3\alpha_1^2=3$ & No, by $(\bmod \, 4)$\\
\hline & $(15,4)$ & $30\beta_1^2-4\alpha_1^2=3$ & No, by parity\\
\hline & $(12,5)$ & $24\beta_1^2-5\alpha_1^2=3$ & No, by $(\bmod \, 5)$\\
\hline
\end{tabular}
\newline \small{Table 8: List out all Pell-type equations \eqref{pell-1-2} with solutions when $D_{31} = 3$, $D_{32} = 5$.}
\end{center}

Case 4: $D_{31} = 4$ and $D_{32} = 3$. Then equation \eqref{km-2} gives $k m = 6 d_a d_b$. 

\begin{center}
\begin{tabular}{|c|c|c|c|}
\hline $(d_a, d_b)$ & $(k, m)$ as $k > \frac{d_a}{d_b} m$ & Pell equation \eqref{pell-1-2} & Solution? \\
\hline $(1,1)$ & $(6,1)$ & $6\beta_1^2-\alpha_1^2=8$ & $\beta_1=2$ and $\alpha_1=4$\\
\hline & $(3,2)$ & $3\beta_1^2-2\alpha_1^2=8$ & No, by $(\bmod \, 3)$\\
\hline $(2,1)$ & $(12,1)$ & $12\beta_1^2-2\alpha_1^2=8$ & No, by $(\bmod \, 3)$\\
\hline & $(6,2)$ & $6\beta_1^2-4\alpha_1^2=8$ & $\beta_1=2$ and $\alpha_1=2$\\
\hline & $(4,3)$ & $4\beta_1^2-6\alpha_1^2=8$ & No, by $(\bmod \, 3)$\\
\hline & $(3,4)$ & $3\beta_1^2-8\alpha_1^2=8$ & No, by $(\bmod \, 3)$\\
\hline $(1,2)$ & $(12,1)$ & $24\beta_1^2-\alpha_1^2=8$ & $\beta_1=1$ and $\alpha_1=4$\\
\hline & $(6,2)$ & $12\beta_1^2-2\alpha_1^2=8$ & No, by $(\bmod \, 3)$\\
\hline
\end{tabular}
\newline \small{Table 9: List out all Pell-type equations \eqref{pell-1-2} with solutions when $D_{31} = 4$, $D_{32} = 3$.}
\end{center}

Case 5: $D_{31} = 4$ and $D_{32} = 5$. Then equation \eqref{km-2} gives $k m = 30 d_a d_b$.

\begin{center}
\begin{tabular}{|c|c|c|c|}
\hline $(d_a, d_b)$ & $(k, m)$ as $k > \frac{d_a}{d_b} m$ & Pell equation \eqref{pell-1-2} & Solution? \\
\hline $(1,1)$ & $(30,1)$ & $30\beta_1^2-\alpha_1^2=8$ & No, by $(\bmod \, 5)$\\
\hline & $(15,2)$ & $15\beta_1^2-2\alpha_1^2=8$ & No, by $(\bmod \, 3)$\\
\hline & $(10,3)$ & $10\beta_1^2-3\alpha_1^2-8$ & No, by $(\bmod \, 3)$\\
\hline & $(6,5)$ & $6\beta_1^2-5\alpha_1^2=8$ & No, by $(\bmod \, 3)$\\
\hline $(2,1)$ & $(60,1)$ & $60\beta_1^2-2\alpha_1^2=8$ & No by,  $(\bmod \, 3)$\\
\hline & $(30,2)$ & $30\beta_1^2-4\alpha_1^2=8$ & No, by $(\bmod \, 5)$\\
\hline & $(20,3)$ & $20\beta_1^2-6\alpha_1^2=8$ & No, by $(\bmod \, 5)$\\
\hline & $(15,4)$ & $15\beta_1^2-8\alpha_!^2=8$ & No, by $(\bmod \, 3)$\\
\hline & $(12,5)$ & $12\beta_1^2-10\alpha_1^2=8$ & $\beta_1=2$ and $\alpha_1=2$\\
\hline & $(10,6)$ & $10\beta_1^2-12\alpha_1^2=8$ & No, by $(\bmod \, 3)$\\
\hline & $(6,10)$ & $6\beta_1^2-20\alpha_1^2=8$ & No, by $(\bmod \, 3)$\\
\hline $(1,2)$ & $(60,1)$ & $120\beta_1^2-\alpha_1^2=8$ & No, by $(\bmod \, 5)$\\
\hline & $(30,2)$ & $60\beta_1^2-2\alpha_1^2=8$ & No, by $(\bmod \, 3)$\\
\hline & $(20,3)$ & $40\beta_1^2-3\alpha_1^2=8$ & No, by $(\bmod \, 3)$\\
\hline & $(15,4)$ & $30\beta_1^2-4\alpha_1^2=8$ & No, by $(\bmod \, 5)$\\
\hline & $(12,5)$ & $24\beta_1^2-5\alpha_1^2=8$ & No, by $(\bmod \, 5)$\\
\hline
\end{tabular}
\newline \small{Table 10: List out all Pell-type equations \eqref{pell-1-2} with solutions when $D_{31} = 4$, $D_{32} = 5$.}
\end{center}



\section{The situation when $D_{21} = 3$}
Suppose $D_{21} = 3$. We can restrict our attention to $4 \le D_{31} \le 5$ by Lemmas \ref{lem4} and \ref{lem6}. Also, Lemmas \ref{lem5} and \ref{lem8} tells us that $D_{32}\geq D_{31} - 2$ and $D_{31} \neq D_{32}$. Then, equations \eqref{km-g} and \eqref{pell-1-g} give
\[
km = \frac{d_a d_b}{3}{D_{32} D_{31} (D_{32} + 3 - D_{31})},
\]
and
\begin{equation} \label{pell-3}
d_b k \beta_1^2 - d_a m \alpha_1^2 = D_{31} (D_{31} - 3).
\end{equation}
Note that $d_a d_b \mid 3$ by \eqref{start1k-g} and we have the following five cases by Lemma \ref{lem7}.

\bigskip

Case 1: $D_{31}=4$ and $D_{32}=2$. We have $k m = \frac{8 d_a d_b}{3}$ and $d_a = d_b = 1$ is impossible.

\begin{center}
\begin{tabular}{|c|c|c|c|}
\hline $(d_a, d_b)$ & $(k, m)$ as $k > \frac{d_a}{d_b} m$ & Pell equation \eqref{pell-3} & Solution? \\
\hline $(3,1)$ & $(8,1)$ & $8\beta_1^2-3\alpha_1^2=4$ & No, by $(\bmod \, 3)$\\
\hline & $(4,2)$ & $4\beta_1^2-6\alpha_1^2=4$ & $\beta_1=5$ and $\alpha_1=4$\\
\hline & $(2,4)$ & $2\beta_1^2-12\alpha_1^2=4$ & No, by $(\bmod \, 3)$\\
\hline $(1,3)$ & $(8,1)$ & $24\beta_1^2-\alpha_1^2=4$ & No, by $(\bmod \, 3)$\\
\hline
\end{tabular}
\newline \small{Table 11: List out all Pell-type equations \eqref{pell-3} with solutions when $D_{31} = 4$, $D_{32} = 2$.}
\end{center}

Case 2: $D_{31}=4$ and $D_{32}=3$. We have $km=8 d_a d_b$.

\begin{center}
\begin{tabular}{|c|c|c|c|}
\hline $(d_a, d_b)$ & $(k, m)$ as $k > \frac{d_a}{d_b} m$ & Pell equation \eqref{pell-3} & Solution? \\
\hline $(3,1)$ & $(24,1)$ & $24\beta_1^2-3\alpha_1^2=4$ & No, by $(\bmod \, 3)$\\
\hline &$(12,2)$ & $12\beta_1^2-6\alpha_1^2=4$ & No, by $(\bmod \, 3)$\\
\hline & $(8,3)$ & $8\beta_1^2-9\alpha_1^2=4$ & No, by $(\bmod \, 3)$\\
\hline & $(6,4)$ & $6\beta_1^2-12\alpha_1^2=4$ & No, by $(\bmod \, 3)$\\
\hline & $(4,6)$ & $4\beta_1^2-18\alpha_1^2=4$ & $\beta_1=17$ and $\alpha_1=8$\\
\hline & $(3,8)$ & $3\beta_1^2-24\alpha_1^2=4$ & No, by $(\bmod \, 3)$\\
\hline $(1,3)$ & $(24,1)$ & $72\beta_1^2-\alpha_1^2=4$ & No, by $(\bmod \, 3)$\\
\hline & $(12,2)$ & $36\beta_1^2-2\alpha_1^2=4$ & $\beta_1=1$ and $\alpha_1=4$\\
\hline $(1,1)$ & $(8,1)$ & $8\beta_1^2-\alpha_1^2=4$ & $\beta_1=1$ and $\alpha_1=2$\\
\hline & $(4,2)$ & $4\beta_1^2-2\alpha_1^2=4$ & $\beta_1=3$ and $\alpha_1=4$\\
\hline
\end{tabular}
\newline \small{Table 12: List out all Pell-type equations \eqref{pell-3} with solutions when $D_{31} = 4$, $D_{32} = 3$.}
\end{center}

Case 3: $D_{31} = 4$ and $D_{32} = 5$. We have $k m = \frac{80}{3} d_a d_b$ and $d_a = d_b = 1$ is impossible.

\begin{center}
\begin{tabular}{|c|c|c|c|}
\hline $(d_a, d_b)$ & $(k, m)$ as $k > \frac{d_a}{d_b} m$ & Pell equation \eqref{pell-3} & Solution? \\
\hline $(3,1)$ & $(80,1)$ & $80\beta_1^2-3\alpha_1^2=4$ & No, by $(\bmod \, 3)$\\
\hline & $(40,2)$ & $40\beta_1^2-6\alpha_1^2=4$ & No, by $(\bmod \, 8)$\\
\hline & $(20,4)$ & $20\beta_1^2-12\alpha_1^2=4$ & No, $(\bmod \, 3)$\\
\hline & $(16,5)$ & $16\beta_1^2-15\alpha_1^2=4$ & $\beta_1=2$ and $\alpha_1=2$\\
\hline & $(10,8)$ & $10\beta_1^2-24\alpha_1^2=4$ & No, by $(\bmod \, 4)$\\
\hline $(1,3)$ & $(80,1)$ & $240\beta_1^2-\alpha_1^2=4$ & No, by $(\bmod \, 3)$\\
\hline & $(40,2)$ & $120\beta_1^2-2\alpha_1^2=4$ & No, by $(\bmod \, 4)$\\
\hline & $(20,4)$ & $60\beta_1^2-4\alpha_1^2=4$ & No, by $(\bmod \, 3)$\\
\hline & $(16,5)$ & $48\beta_1^2-5\alpha_1^2=4$ & No, by $(\bmod \, 5)$\\
\hline
\end{tabular}
\newline \small{Table 13: List out all Pell-type equations \eqref{pell-3} with solutions when $D_{31} = 4$, $D_{32} = 5$.}
\end{center}

Case 4: $D_{31} = 5$ and $D_{32} = 3$. We have $k m = 5 d_a d_b$.

\begin{center}
\begin{tabular}{|c|c|c|c|}
\hline $(d_a, d_b)$ & $(k, m)$ as $k > \frac{d_a}{d_b} m$ & Pell equation \eqref{pell-3} & Solution? \\
\hline $(1,1)$ & $(5,1)$ & $5\beta_1^2-\alpha_1^2=10$ & No, by $(\bmod \, 4)$\\
\hline $(3,1)$ & $(15,1)$ & $15\beta_1^2-3\alpha_1^2=10$ & No, by $(\bmod \, 3)$\\
\hline & $(5,3)$ & $5\beta_1^2-9\alpha_1^2=10$ & No, by $(\bmod \, 3)$\\
\hline & $(3,5)$ & $3\beta_1^2-15\alpha_1^2=10$ & No, by $(\bmod \, 3)$\\
\hline $(1,3)$ & $(15,1)$ & $45\beta_1^2-\alpha_1^2=10$ & No, by $(\bmod \, 3)$\\
\hline
\end{tabular}
\newline \small{Table 14: List out all Pell-type equations \eqref{pell-3} with solutions when $D_{31} = 5$, $D_{32} = 3$.}
\end{center}



\section{The situation when $D_{21} = 4$}

Suppose $D_{21} = 4$. We may restrict our attention to $D_{31} = 5$ by Lemmas \ref{lem4} and \ref{lem6}. Also, Lemmas \ref{lem5} and \ref{lem8} tells us that $D_{32}\geq D_{31} - 3$. Then, equations \eqref{km-g} and \eqref{pell-1-g} give
\begin{equation} \label{4start1k}
\left\{ \begin{array}{l} \alpha_2 \beta_1 - \beta_2 \alpha_1 = \frac{4}{d_a d_b}, \\
d_b a_3 \beta_1 - d_a b_3 \alpha_1 = D_{31},
\end{array} \right.
\end{equation}
\begin{equation*} \label{km-4}
k m = \frac{d_a d_b}{4}{D_{32} D_{31} (D_{32} - D_{31} + 4)},
\end{equation*}
and
\begin{equation} \label{pell-4}
d_b k \beta_1^2 - d_a m \alpha_1^2 = D_{31} (D_{31} - 4).
\end{equation}
\noindent We have the following four cases by Lemma \ref{lem7}. Note that $d_a d_b \mid 4$ by \eqref{4start1k}.

\bigskip

Case 1: $D_{31} = 5$ and $D_{32} = 2$. We have $k m = \frac{5}{2} d_a d_b$ and $d_a = d_b = 1$ is impossible. We apply Lemma \ref{lem-Pell-check-1} with $p = 5$ for one of the checks below.

\begin{center}
\begin{tabular}{|c|c|c|c|}
\hline $(d_a, d_b)$ & $(k, m)$ as $k > \frac{d_a}{d_b} m$ & Pell equation \eqref{pell-4} & Solution? \\
\hline $(2,2)$ & $(k, m)$ & $2 k \beta_1^2- 2 m \alpha_1^2=5$ & No, by parity\\
\hline $(2,1)$ & $(5,1)$ & $5\beta_1^2-2\alpha_1^2=5$ & $\beta_1=19, \; \; \alpha_1=30$\\
\hline $(1,2)$ & $(5,1)$ & $10\beta_1^2-\alpha_1^2=5$ & No, by Lemma \ref{lem-Pell-check-2} \\
\hline & $(5,2)$ & $10\beta_1^2-4\alpha_1^2=5$ & No, by parity\\
\hline $(4,1)$ & $(10,1)$ & $10\beta_1^2-4\alpha_1^2=5$ & No, by parity\\
\hline & $(5,2)$ & $5\beta_1^2-8\alpha_1^2=5$ & $\beta_1=19, \; \; \alpha_1=15$\\
\hline & $(2,5)$ & $2\beta_1^2-20\alpha_1^2=5$ & No, by parity\\
\hline $(1,4)$ & $(10,1)$ & $40\beta_1^2-\alpha_1^2=5$ & No, by $(\bmod \, 4)$\\
\hline
\end{tabular}
\newline \small{Table 15: List out all Pell-type equations \eqref{pell-4} with solutions when $D_{31} = 5$, $D_{32} = 2$.}
\end{center}

Case 2: $D_{31} = 5$ and $D_{32} = 3$. We have $k m = \frac{15}{2} d_a d_b$ and $d_a = d_b = 1$ is impossible. We apply Lemma \ref{lem-Pell-check-1} with $p = 5$ for one of the checks below.

\begin{center}
\begin{tabular}{|c|c|c|c|}
\hline $(d_a, d_b)$ & $(k, m)$ as $k > \frac{d_a}{d_b} m$ & Pell equation \eqref{pell-4} & Solution? \\
\hline $(2,2)$ & $(k,m)$ & $2 k \beta_1^2-2 m \alpha_1^2=5$ & No, by parity \\
\hline $(1,2)$ & $(15,1)$ & $30\beta_1^2-\alpha_1^2=5$ & $\beta_1=1, \; \; \alpha_1=5$\\
\hline $(2,1)$ & $(15,1)$ & $15\beta_1^2-2\alpha_1^2=5$ & No, by $(\bmod \, 3)$\\
\hline & $(5,3)$ & $5\beta_1^2-6\alpha_1^2=5$ & $\beta_1=11, \; \; \alpha_1=10$\\
\hline & $(3,5)$ & $3\beta_1^2-10\alpha_1^2=5$ & No, by Lemma \ref{lem-Pell-check-1} \\
\hline $(4,1)$ & $(30,1)$ & $30\beta_1^2-4\alpha_1^2=5$ & No, by parity\\
\hline & $(15,2)$ & $15\beta_1^2-8\alpha_1^2=5$ & No, by $(\bmod \, 3)$\\
\hline & $(10,3)$ & $10\beta_1^2-12\alpha_1^2=5$ & No, by parity\\
\hline & $(6,5)$ & $6\beta_1^2-20\alpha_1^2=5$ & No, by parity\\
\hline & $(3,10)$ & $3\beta_1^2-40\alpha_1^2=5$ & No, by parity\\
\hline & $(5,6)$ & $5\beta_1^2-24\alpha_1^2=5$ & $\beta_1=11, \; \; \alpha_1=5$\\
\hline $(1,4)$ & $(30,1)$ & $120\beta_1^2-\alpha_1^2=5$ & No, by $(\bmod \, 4)$\\
\hline & $(15,2)$ & $60\beta_1^2-2\alpha_1^2=5$ & No, by parity\\
\hline
\end{tabular}
\newline \small{Table 16: List out all Pell-type equations \eqref{pell-4} with solutions when $D_{31} = 5$, $D_{32} = 3$.}
\end{center}

\section{A formula for the ratio $R_4$}

Our ultimate goal is to estimate the ratio $R_4 = A / a_3^3$. To achieve this, we express everything in terms of $\alpha_1$ (as appeared in \eqref{pell-1-g}).
\begin{lem} \label{lem14}
For $1 \le D_{21}, D_{31}, D_{32} \le 5$, we have
\begin{equation*}
    R_4 = \frac{m (D_{31}-D_{21}) \bigl( D_{21}\sqrt{\frac{k m}{d_a d_b}}+D_{32}{D_{31} \bigr)\bigl[ D_{21} + D_{31}\sqrt{\frac{d_a d_b}{k m}}(D_{32}-D_{31}+D_{21})}\bigr]}
    {D_{21}D_{31}^2 d_a\bigl( \sqrt{\frac{k m}{d_a d_b}}+D_{32} \bigr)^3} + O\Bigl(\frac{1}{\alpha_1^2}\Bigr).
\end{equation*}
\end{lem}

\begin{proof}
Since $d_a d_b = \gcd(a_1, a_2) \gcd(b_1, b_2) \mid a_2 b_1 - b_2 a_1 = D_{21} \le 5$, we have $1 \le d_a d_b \le 5$. Combining this with \eqref{km-g} and Lemmas \ref{lem4} and \ref{lem5}, we have $1 \le k m \le 5 \cdot 5 \cdot 5 \cdot 4 = 500$. From the Pell-type equation \eqref{pell-1-g}, we obtain
\[
\Bigl( \beta_1 - \sqrt{\frac{d_a m}{d_b k}} \alpha_1 \Bigr) \Bigl( \beta_1 + \sqrt{\frac{d_a m}{d_b k}} \alpha_1 \Bigr) = \beta_1^2 - \frac{d_a m}{d_b k} \alpha_1^2 = O(1)
\]
which implies
\begin{equation} \label{lem10}
\frac{\beta_1}{\alpha_1} - \sqrt{\frac{d_a m}{d_b k}} = O \Bigl( \frac{1}{\sqrt{(d_a m) / (d_b k)} \alpha_1^2} \Bigr) = O \Bigl( \frac{1}{\alpha_1^2} \Bigr)
\end{equation}
as $\frac{d_b k}{d_a m} \le \frac{5 \cdot 500}{1 \cdot 1}$. 

\bigskip

Next, we find formulas for $a_3$, $A$, and $R_4$. Substituting \eqref{alpha2beta2} into \eqref{a3b3}, we get
\begin{equation} \label{a3-new}
a_3 = \frac{k \beta_1 + d_a D_{32} \alpha_1}{D_{31} - D_{21}}.
\end{equation}
By \eqref{basic3'}, \eqref{definitions} and \eqref{a3-new}, we have $A = \frac{a_2 a_1 (b_2 - b_1)}{D_{21}} = \frac{d_a^2 d_b \alpha_2 \alpha_1 (\beta_2 - \beta_1)}{D_{21}}$. Hence, we obtain
\begin{equation} \label{R4-temp1}
R_4 = \frac{A}{a_3^3} = \frac{(D_{31} - D_{21})^3 d_a^2 d_b \alpha_2 \alpha_1 (\beta_2 - \beta_1)}{D_{21} (k \beta_1 + d_a D_{32} \alpha_1)^3}.
\end{equation}
From \eqref{alpha2beta2}, we have $\beta_2 - \beta_1 = \frac{D_{21} m \alpha_1 + d_b D_{31} (D_{32} + D_{21} - D_{31}) \beta_1}{d_b D_{31} (D_{31} - D_{21})}$. Substituting this and \eqref{alpha2beta2} into \eqref{R4-temp1}, we obtain
\begin{align*}
R_4 =& \frac{(D_{31} - D_{21}) d_a \alpha_1 (D_{21} k \beta_1 + d_a D_{31} D_{32} \alpha_1) (D_{21} m \alpha_1 + d_b D_{31} (D_{32} + D_{21} - D_{31}) \beta_1)}{D_{21} D_{31}^2 (k \beta_1 + d_a D_{32} \alpha_1)^3} \\
=& \frac{(D_{31} - D_{21}) d_a \bigl( D_{21} k \frac{\beta_1}{\alpha_1} + d_a D_{31} D_{32} \bigr) \bigl(D_{21} m + d_b D_{31} (D_{32} + D_{21} - D_{31}) \frac{\beta_1}{\alpha_1} \bigr)}{D_{21} D_{31}^2 (k \frac{\beta_1}{\alpha_1} + d_a D_{32})^3}.
\end{align*}
Finally, we apply \eqref{lem10} to the above equation and get
\begin{align*}
R_4 =& \frac{(D_{31} - D_{21}) d_a^2 \bigl( D_{21} \sqrt{\frac{k m}{d_a d_b}} + D_{31} D_{32} + O( \frac{1}{\alpha_1^2} ) \bigr)}{D_{21} D_{31}^2 d_a^3 \bigl(\sqrt{\frac{k m}{d_a d_b}} + D_{32} + O( \frac{1}{\alpha_1^2} ) \bigr)^3} \\
&\cdot  \Bigl(D_{21} m + D_{31} (D_{32} + D_{21} - D_{31}) \sqrt{\frac{d_a d_b m}{k}} + O \Bigl( \frac{1}{\alpha_1^2} \Bigr) \Bigr) \\
=& \frac{m (D_{31} - D_{21}) \bigl( D_{21} \sqrt{\frac{k m}{d_a d_b}} + D_{31} D_{32} \bigr) \bigl( 1 + O( \frac{1}{\alpha_1^2} ) \bigr)}{D_{21} D_{31}^2 d_a \bigl(\sqrt{\frac{k m}{d_a d_b}} + D_{32} \bigr)^3 \bigl( 1 + O( \frac{1}{\alpha_1^2} ) \bigr)^3} \\
&\cdot  \Bigl[ D_{21} + D_{31} (D_{32} + D_{21} - D_{31}) \sqrt{\frac{d_a d_b}{k m}} \Bigr] \Bigl(1 + O \Bigl( \frac{1}{\alpha_1^2} \Bigr)  \Bigr).
\end{align*}
This gives the lemma.
\end{proof}


\section{Proof of Theorem \ref{thm1}}

\begin{proof}
If one of $D_{21}, D_{31}, D_{32}$ is greater than $5$, then Lemma \ref{lem6} with $C = a_3$ implies $R_4 < 0.042$ which is smaller than the upper bound in the theorem when $n$ is sufficiently large. By Lemma \ref{lem7}, we have $R_4 \le 0.04$ when $D_{31} = 5$ and $D_{32} = 4$. If $D_{31} = D_{32}$, then Lemma \ref{lem8} implies $R_4 \le 0.25 / a_3$ which is much smaller than $0.042$ when $n$ (and, hence, $a_3$) is large. So, we can narrow our attention to $1 \le D_{21}, D_{31}, D_{32} \le 5$ with $D_{31} \neq D_{32}$ and omitting the case $D_{31} = 5$ and $D_{32} = 4$. Based on the previous sections, we see that the four close factorizations of $n$ imply an integer solution to a certain Pell-type equation \eqref{pell-1-g}. However, as shown from the tables in the previous sections, many selections of the parameters $D_{21}$, $D_{31}$, $D_{32}$, $d_a$, $d_b$, $k$, and $m$ yield no integer solution. Thus, we can focus on those Pell-type equations with solutions and we summarize them into the table below, using Lemma \ref{lem14} to compute $R_4$ (ignoring error term). 

\bigskip

One can see that the largest ratio (in red and the only one bigger than $0.042$) comes from the situation when $D_{21} = 1$, $D_{31} = 3$, $D_{32} = 4$, $d_a = d_b = 1$, and $(k, m) = (6, 4)$. Putting these parameters into Lemma \ref{lem14} and simplifying, we arrive at
\[
R_4 = \frac{6 + \sqrt{6}}{9 (2 + \sqrt{6})^2} + O\Bigl(\frac{1}{\alpha_1^2}\Bigr).
\]
From \eqref{a3-new} and Lemma \ref{lem10}, we know that $a_3 = O(\alpha_1)$ as $d_a \le 5$ and $m \le 500$. Hence, it follows that $\sqrt{n} \le A \le a_2 a_1 (b_2 - b_1) < a_3^3 = O( \alpha_1^3 )$ by \eqref{basic3'} and Lemma \ref{lem1}. This implies $\frac{1}{\alpha_1^3} = O( \frac{1}{\sqrt{n}} )$ or $\frac{1}{\alpha_1^2} = O( \frac{1}{n^{1/3}} )$, and, hence, Theorem \ref{thm1}.
\begin{center}
\begin{tabular}{|c|c|c|c|c|c|c|c|}
\hline
\multicolumn{3}{|c|}{$D_{ij}$ Skews} & \multicolumn{2}{c|}{Divisors} & \multicolumn{2}{c|}{Parameters} & Ratio $R_4 = A / a_3^3$\\
\hline
$D_{21}$ & $D_{31}$ & $D_{32}$ & $d_a$ & $d_b$ & $k$ & $m$ & \\ 
\hline
\hline
1 & 2 & 4 & 1 & 1 & 6 & 4 & 0.04072067323 \\
1 & 2 & 5 & 1 & 1 & 20 & 2 & 0.01272913946 \\
1 & 3 & 4 & 1 & 1 & 6 & 4 & {\color{red}0.04742065558 } \\
\hline
\multirow{2}{*}{1} & \multirow{2}{*}{4} & \multirow{2}{*}{5} & \multirow{2}{*}{1} & \multirow{2}{*}{1} & 20 & 2 & 0.01539501058 \\
 & & & & & 8 & 5 & 0.03848752646 \\
\hline
\multirow{3}{*}{2} & \multirow{3}{*}{3} & \multirow{3}{*}{2} & 1 & 1 & 3 & 1 & 0.03774955135 \\
 & & & 2 & 1 & 3 & 2 & 0.03774955135 \\
 & & & 1 & 2 & 6 & 1 & 0.03774955135 \\
\hline
2 & 3 & 4 & 1 & 1 & 6 & 3 & 0.02512626585 \\
\hline
2 & 4 & 5 & 2 & 1 & 12 & 5 & 0.01762424561 \\
\hline
\multirow{2}{*}{2} & \multirow{2}{*}{5} & \multirow{2}{*}{4} & 1 & 1 & 10 & 1 & 0.01539501058 \\
 & & & 1 & 2 & 20 & 1 & 0.01539501058 \\
\hline
3 & 4 & 2 & 3 & 1 & 4 & 2 & 0.02036033661 \\
\hline
\multirow{4}{*}{3} & \multirow{4}{*}{4} & \multirow{4}{*}{3} & 3 & 1 & 4 & 6 & 0.02512626585 \\
 & & & 1 & 3 & 12 & 2 & 0.02512626585 \\
 & & & 1 & 1 & 8 & 1 & 0.01256313292 \\
 & & & 1 & 1 & 4 & 2 & 0.02512626585 \\
\hline
3 & 4 & 5 & 3 & 1 & 16 & 5 & 0.00715749421 \\
\hline
\multirow{2}{*}{4} & \multirow{2}{*}{5} & \multirow{2}{*}{2} & 2 & 1 & 5 & 1 & 0.01272913946 \\
 & & & 4 & 1 & 5 & 2 & 0.01272913946 \\
\hline
\multirow{2}{*}{4} & \multirow{2}{*}{5} & \multirow{2}{*}{3} & 2 & 1 & 5 & 3 & 0.01576260533 \\
 & & & 1 & 2 & 15 & 1 & 0.01050840355 \\
\hline
\end{tabular}
\newline \small{Table 17: List out all cases from Tables 1 - 16 with solutions and compute $R_4$.}
\end{center}

\bigskip

The generalized Pell equation \eqref{pell-1-g} corresponding to the largest ratio is
\[
6 \beta_1^2 - 4 \alpha_1^2 = 6 \; \; \Leftrightarrow \; \; 3 \beta_1^2 - 2 \alpha_1^2 = 3 \; \; \Leftrightarrow \; \; \beta_1^2 - 6 \Bigl( \frac{\alpha_1}{3} \Bigr)^2 = 1
\]
Note: One can see that $3 \mid \alpha_1$ and, hence, $\frac{\alpha_1}{3}$ is an integer. By the theory of Pell equation, all integer solutions are generated by the fundamental solution: For integer $i \ge 1$, we have
\[
\beta_1 + \frac{\alpha_1}{3} \sqrt{6} = (5 + 2 \sqrt{6})^i =: x_i + y_i \sqrt{6}.
\]
With $b_1 = \beta_1 = x_i$ and $a_1 = \alpha_1 = 3 y_i$, we can construct the following example:
\[
a_1 = 3 y_i, \; \; b_1 = x_i, \; \; a_2 = x_i + 6 y_i, \; \; b_2 = 2 x_i + 2 y_i, \; \; a_3 = 3 x_i + 6 y_i, \; \; b_3 = 2 x_i + 6 y_i,
\]
\[
A = 3 y_i (x_i + 2 y_i) (x_i + 6 y_i), \; \; B = 2 x_i (x_i + y_i) (x_i + 3 y_i), \; \; \text{ and } \; \; n = A B
\]
using \eqref{basic3}, \eqref{basic3'}, \eqref{definitions}, \eqref{a3b3} and \eqref{alpha2beta2}. One can check that
\[
n = A B = (A + a_1) (B - b_1) = (A + a_2) (B - b_2) = (A + a_3) (B - b_3)
\]
and
\begin{align*}
\frac{A}{a_3^3} =& \frac{3 y_i (x_i + 2 y_i) (x_i + 6 y_i)}{(3 x_i + 6 y_i)^3} = \frac{y_i (x_i + 6 y_i)}{9 (x_i + 2 y_i)^2} = \frac{y_i \bigl( (6 + \sqrt{6}) y_i + O(\frac{1}{y_i}) \bigr)}{9 \bigl( (2 + \sqrt{6}) y_i + O(\frac{1}{y_i}) \bigr)^2} \\
=& \frac{ 6 + \sqrt{6} + O( \frac{1}{y_i^2} )}{9 \bigl( 2 + \sqrt{6} + O( \frac{1}{y_i^2} ) \bigr)^2} = \frac{6 + \sqrt{6}}{9 (2 + \sqrt{6})^2} + O \Bigl( \frac{1}{\alpha_1^2} \Bigr) = \frac{6 + \sqrt{6}}{9 (2 + \sqrt{6})^2} + O \Bigl( \frac{1}{n^{1/3}} \Bigr)
\end{align*}
by $x_i - y_i \sqrt{6} = \frac{1}{x_i + y_i \sqrt{6}}$ or $x_i = \sqrt{6} y_i + O(\frac{1}{y_i})$. This shows that the upper bound in Theorem \ref{thm1} is best possible.
\end{proof}


\end{document}